\documentclass[preprint,10pt]{elsarticle}

\usepackage{amsmath}
\usepackage{amssymb}
\usepackage{amsthm}

\newtheorem{thm}{Theorem}[section]

\newtheorem{lem}[thm]{Lemma}

\newtheorem{defin}[thm]{Definition}
\newtheorem{rem}[thm]{Remark}

\numberwithin{equation}{section}

\begin{document}

\begin{frontmatter}

\title{Homogeneous fractional integral operators on weighted Lebesgue, Morrey and Campanato spaces}

\author{Jingliang Du and Hua Wang
\footnote{E-mail address: 418419773@qq.com, wanghua@pku.edu.cn.}}
\address{School of Mathematics and System Sciences, Xinjiang University,\\
Urumqi 830046, P. R. China\\
\textbf{Dedicated to the memory of Li Xue}}

\begin{abstract}
Let $0<\alpha<n$ and $T_{\Omega,\alpha}$ be the homogeneous fractional integral operator which is defined by
\begin{equation*}
T_{\Omega,\alpha}f(x):=\int_{\mathbb R^n}\frac{\Omega(x-y)}{|x-y|^{n-\alpha}}f(y)\,dy,
\end{equation*}
where $\Omega$ is homogeneous of degree zero in $\mathbb R^n$ for $n\geq2$, and is integrable on the unit sphere $\mathbb{S}^{n-1}$. In this paper we study boundedness properties of the homogeneous fractional integral operator $T_{\Omega,\alpha}$ acting on weighted Lebesgue and Morrey spaces. Under certain Dini-type smoothness condition on $\Omega$, we prove that $T_{\Omega,\alpha}$ is bounded from $L^{p}(\omega^p)$ to $\mathcal{C}^{\gamma,\ell}_{\omega}$(a class of Campanato spaces) for appropriate indices, when $n/{\alpha}<p<\infty$. Moreover, we prove that if $\Omega$ satisfies certain Dini-type smoothness condition on $\mathbb{S}^{n-1}$, then $T_{\Omega,\alpha}$ is bounded from $\mathcal{M}^{p,\kappa}(\omega^p,\omega^q)$ to $\mathcal{C}^{\gamma,\ell}(\omega^q)$(weighted Campanato spaces) for appropriate indices, when $p/q<\kappa<1$.
\end{abstract}

\begin{keyword}
Homogeneous fractional integral operators, weighted Morrey spaces, weighted Campanato spaces, Dini-type smoothness condition
\MSC[2020] Primary 42B20, 42B25; Secondary 42B35
\end{keyword}

\end{frontmatter}

\section{Introduction and preliminaries}
\label{sec1}
The symbols $\mathbb R$ and $\mathbb N$ stand for the sets of all real numbers and natural numbers, respectively. Let $n\in\mathbb N$ and $\mathbb R^n$ be the $n$-dimensional Euclidean space endowed with the Lebesgue measure $dx$. The Euclidean norm of $x=(x_1,x_2,\dots,x_n)\in \mathbb R^n$ is given by $|x|:=(\sum_{i=1}^n x_i^2)^{1/2}$.
One of the most significant operators in harmonic analysis is the Riesz potential operator $I_{\alpha}$ of order $\alpha$ $(0<\alpha<n)$, also known as the fractional integral operator. Given $0<\alpha<n$, the Riesz potential operator $I_{\alpha}$ of order $\alpha$ is defined by
\begin{equation*}
I_{\alpha}f(x):=\frac{1}{\gamma(\alpha)}\int_{\mathbb R^n}\frac{f(y)}{|x-y|^{n-\alpha}}dy, \quad x\in\mathbb R^n,
\end{equation*}
where
\begin{equation*}
\gamma(\alpha):=\frac{\pi^{n/2}2^{\alpha}\Gamma(\alpha/2)}{\Gamma({(n-\alpha)}/2)}
\end{equation*}
and $\Gamma(\cdot)$ being the usual gamma function (see \cite{grafakos}). Let $(-\Delta)^{\alpha/2}$ denote the $\alpha/2$-th order Laplacian. Then $u=I_{\alpha}f$ is viewed as a solution of the $\alpha/2$-th order Laplace equation
\begin{equation*}
(-\Delta)^{\alpha/2}u=f
\end{equation*}
in the sense of the Fourier transform; i.e., $(-\Delta)^{\alpha/2}$ exists as the inverse of $I_{\alpha}$. It is well known that the Riesz potential operator $I_{\alpha}$ plays an important role in harmonic analysis, potential theory and PDEs, particularly in the study of smoothness properties of functions. Let $0<\alpha<n$ and $1<p<q<\infty$. The classical Hardy--Littlewood--Sobolev inequality states that $I_{\alpha}$ is bounded from $L^p(\mathbb R^n)$ to $L^q(\mathbb R^n)$ if and only if $1/q=1/p-\alpha/n$. Historically, Sobolev proved the $n$-dimensional version of the Hardy--Littlewood--Sobolev inequality using the one-dimensional version which is due to Hardy and Littlewood. Note that the range of $p$ is $1<p<n/{\alpha}$. However, for the critical index $p=n/{\alpha}$ and $0<\alpha<n$, it is easy to verify that the operator $I_{\alpha}$ is not bounded from $L^{n/{\alpha}}(\mathbb R^n)$ to $L^{\infty}(\mathbb R^n)$ (see \cite[p.119]{stein}). Here $L^{\infty}(\mathbb R^n)$ denotes the Banach space of all essentially bounded measurable functions $f$ on $\mathbb R^n$. Instead, in this case, we have that as a substitute the Riesz potential operator $I_{\alpha}$ is bounded from $L^{n/{\alpha}}(\mathbb R^n)$ to $\mathrm{BMO}(\mathbb R^n)$ (see \cite[p.130]{duoand}).

In this paper, we are concerned with the boundedness properties of fractional integral operators with homogeneous kernels.
Let $n\geq2$ and $\mathbb{S}^{n-1}:=\{x\in\mathbb R^n:|x|=1\}$ denote the unit sphere in $\mathbb R^n$ equipped with the normalized Lebesgue measure $d\sigma(x')$. Here $x'=x/{|x|}$ for any $x\neq0$. Let $0<\alpha<n$ and $\Omega$ be a homogeneous function of degree zero on $\mathbb R^n$ and $\Omega\in L^s(\mathbb{S}^{n-1})$ with $s\geq1$. Then the homogeneous fractional integral operator $T_{\Omega,\alpha}$ is defined by
\begin{equation*}
T_{\Omega,\alpha}f(x):=\int_{\mathbb R^n}\frac{\Omega(x-y)}{|x-y|^{n-\alpha}}f(y)\,dy,\quad x\in\mathbb R^n.
\end{equation*}
When $\Omega\equiv1$, $T_{\Omega,\alpha}$ is just the Riesz potential operator $I_{\alpha}$. In 1971, Muckenhoupt and Wheeden \cite{muckenhoupt1} proved that the homogeneous fractional integral operator $T_{\Omega,\alpha}$ is bounded from $L^p(\mathbb R^n)$ to $L^q(\mathbb R^n)$, when $\Omega\in L^s(\mathbb{S}^{n-1})$, $s'<p<n/{\alpha}$ and $1/q=1/p-\alpha/n$ (see also \cite[Theorem 2.1]{wang3} for an alternative proof of this result). Moreover, for the critical index $p=n/{\alpha}$ and $0<\alpha<n$, it can be shown that the operator $T_{\Omega,\alpha}$ is bounded from
$L^{n/{\alpha}}(\mathbb R^n)$ to $\mathrm{BMO}(\mathbb R^n)$, provided that $\Omega$ satisfies certain smoothness condition on the unit sphere $\mathbb{S}^{n-1}$ (see \cite[Theorem 1]{ding2} and \cite[Theorem 2.7]{wang3}).

A weight $\omega$ is a nonnegative, locally integrable function on $\mathbb R^n$ that take values in $(0,\infty)$ almost everywhere. Given a weight $\omega$ and a measurable set $E\subset\mathbb R^n$, we denote by $m(E)$ the Lebesgue measure of the set $E$. The complement of the set $E$ is denoted by $E^{\complement}$, and the characteristic function of the set $E$ is denoted by $\chi_{E}$: $\chi_{E}(x)=1$ if $x\in E$ and $\chi_{E}(x)=0$ if $x\notin E$. We use the notation
\begin{equation*}
\omega(E):=\int_{E}\omega(x)\,dx
\end{equation*}
to denote the $\omega$-measure of the set $E$. For given $x_0\in\mathbb R^n$ and $r>0$, $B(x_0,r)=\{x\in\mathbb R^n:|x-x_0|<r\}$ denotes the open ball centered at $x_0$ with radius $r$, $m(B(x_0,r))$ is the Lebesgue measure of the ball $B(x_0,r)$. Given a ball $B$ and $\lambda>0$, $\lambda B$ denotes the ball with the same center as $B$ whose radius is $\lambda$ times that of $B$. For any given $p\in[1,\infty)$, the weighted Lebesgue space with respect to the measure $\omega(x)dx$ is denoted by $L^p(\omega)$, and we set
\begin{equation*}
\|f\|_{L^p(\omega)}:=\bigg(\int_{\mathbb R^n}|f(x)|^p\omega(x)\,dx\bigg)^{1/p}<+\infty.
\end{equation*}
When $\omega\equiv1$, $L^p(\omega)=L^p(\mathbb R^n)$, the usual Lebesgue space.
To state some known results, we first recall some necessary definitions and notations.
\begin{defin}
A locally integrable function $\hbar$ is said to be in the bounded mean oscillation space $\mathrm{BMO}(\mathbb R^n)$, if
\begin{equation*}
\|\hbar\|_{\mathrm{BMO}}:=\sup_{B\subset\mathbb R^n}\frac{1}{m(B)}\int_B|\hbar(x)-\hbar_B|\,dx<+\infty,
\end{equation*}
where $\hbar_B$ stands for the mean value of $\hbar$ over $B$; i.e.,
\begin{equation*}
\hbar_B:=\frac{1}{m(B)}\int_B \hbar(y)\,dy
\end{equation*}
and the supremum is taken over all balls $B$ in $\mathbb R^n$. Modulo constants, the space $\mathrm{BMO}(\mathbb R^n)$ is a Banach space with respect to the norm $\|\cdot\|_{\mathrm{BMO}}$.
For a weight function $\omega$ on $\mathbb R^n$, the weighted version of $\mathrm{BMO}$ is denoted by $\mathrm{BMO}_{\omega}$. We say that $\hbar$ is in the weighted space $\mathrm{BMO}_{\omega}$, if there is a constant $C>0$ such that for any ball $B$ in $\mathbb R^n$,
\begin{equation*}
\Big[\underset{x\in{B}}{\mathrm{ess\,sup}}\,\omega(x)\Big]\cdot
\bigg(\frac{1}{m(B)}\int_B|\hbar(x)-\hbar_B|\,dx\bigg)\leq C<+\infty.
\end{equation*}
The smallest constant $C$ is then taken to be the norm of $\hbar$ in this space, and is denoted by $\|\hbar\|_{\mathrm{BMO}_\omega}$.
\end{defin}
The space $\mathrm{BMO}(\mathbb R^n)$ was first introduced by John and Nirenberg (see \cite{john}). The weighted space $\mathrm{BMO}_{\omega}$ was first introduced by Muckenhoupt and Wheeden (see \cite{muckenhoupt2}).

Following \cite{grafakos} and \cite{muckenhoupt2}, let us give the definitions of some weight classes.

\begin{defin}[\cite{grafakos}]
A weight function $\omega$ defined on $\mathbb R^n$ is in the Muckenhoupt class $A_p$ with $1<p<\infty$, if there exists a positive constant $C$ such that for any ball $B\subset\mathbb R^n$,
\begin{equation*}
\bigg(\frac{1}{m(B)}\int_{B}\omega(x)\,dx\bigg)\bigg(\frac{1}{m(B)}\int_{B}\omega(x)^{1-p'}dx\bigg)^{p-1}\leq C<+\infty,
\end{equation*}
where $p'$ denotes the conjugate exponent of $p>1$ such that $1/{p'}+1/p=1$, and $p'=1$ when $p=\infty$. The smallest constant $C>0$ such that the above inequality holds is called the $A_p$ constant of $\omega$, and is denoted by $[\omega]_{A_p}$. For $p=1$, a weight function $\omega$ defined on $\mathbb R^n$ is in the Muckenhoupt class $A_1$, if there exists a positive constant $C$ such that for any ball $B\subset\mathbb R^n$,
\begin{equation*}
\frac{1}{m(B)}\int_{B}\omega(x)\,dx\leq C\cdot\Big[\underset{x\in B}{\mathrm{ess\,inf}}\,\omega(x)\Big].
\end{equation*}
The infimum of $C>0$ satisfying the above inequality is called the $A_1$ constant of $\omega$, and is denoted by $[\omega]_{A_1}$.
\end{defin}

\begin{defin}[\cite{muckenhoupt2}]
A weight function $\omega$ defined on $\mathbb R^n$ is said to belong to the Muckenhoupt--Wheeden class $A(p,q)$ for $1<p<q<\infty$, if there exists a constant $C>0$ such that
\begin{equation*}
\bigg(\frac{1}{m(B)}\int_{B}\omega(x)^q\,dx\bigg)^{1/q}
\bigg(\frac{1}{m(B)}\int_{B}\omega(x)^{-p'}\,dx\bigg)^{1/{p'}}\leq C<+\infty,
\end{equation*}
for every ball $B$ in $\mathbb R^n$, where $1/{p'}+1/p=1$. When $p=1$, a weight function $\omega$ defined on $\mathbb R^n$ is in the Muckenhoupt--Wheeden class $A(1,q)$ with $1<q<\infty$, if there exists a constant $C>0$ such that
\begin{equation*}
\bigg(\frac{1}{m(B)}\int_{B}\omega(x)^q\,dx\bigg)^{1/q}
\bigg[\underset{x\in B}{\mathrm{ess\,sup}}\,\frac{1}{\omega(x)}\bigg]\leq C<+\infty,
\end{equation*}
for every ball $B$ in $\mathbb R^n$. A weight function $\omega$ defined on $\mathbb R^n$ is said to belong to the Muckenhoupt--Wheeden class $A(p,\infty)$ with $1<p<\infty$, if there exists a constant $C>0$ such that
\begin{equation*}
\Big[\underset{x\in{B}}{\mathrm{ess\,sup}}\,\omega(x)\Big]\cdot
\bigg(\frac{1}{m(B)}\int_{B}\omega(x)^{-p'}\,dx\bigg)^{1/{p'}}\leq C<+\infty,
\end{equation*}
for every ball $B$ in $\mathbb R^n$, where
\begin{equation*}
\underset{x\in B}{\mathrm{ess\,sup}}\,\omega(x):=\inf\big\{M>0:m\big(\big\{x\in B:\omega(x)>M\big\}\big)=0\big\}.
\end{equation*}
\end{defin}

A few historical remarks are given as follows:
\begin{enumerate}
  \item In 1974, Muckenhoupt and Wheeden studied the weighted boundedness of $I_{\alpha}$ for $0<\alpha<n$, and proved that the Riesz potential operator $I_{\alpha}$ is bounded from $L^p(\omega^p)$ to $L^q(\omega^q)$, when $\omega\in A(p,q)$, $1<p<n/{\alpha}$ and $1/q=1/p-{\alpha}/n$ (see \cite[Theorem 2]{muckenhoupt2}). Moreover, for the extreme case $p=n/{\alpha}$, Muckenhoupt and Wheeden also proved the following result (see \cite[Theorem 7]{muckenhoupt2}). Let $0<\alpha<n$ and $p=n/{\alpha}$. If $\omega\in A(p,\infty)$, then the Riesz potential operator $I_{\alpha}$ is bounded from $L^p(\omega^p)$ to $\mathrm{BMO}_{\omega}$.
  \item In 1971, Muckenhoupt and Wheeden \cite{muckenhoupt1} studied weighted norm inequalities for $T_{\Omega,\alpha}$ with the power weight $\omega(x)=|x|^\beta$.In 1998, Ding and Lu considered the general case, and proved that the homogeneous fractional integral operator $T_{\Omega,\alpha}$ is bounded from $L^p(\omega^p)$ to $L^q(\omega^q)$, when $\Omega\in L^s(\mathbb{S}^{n-1})$ with $s>p'$, $\omega^{s'}\in A(p/{s'},q/{s'})$, $1<p<n/{\alpha}$ and $1/q=1/p-{\alpha}/n$, see \cite[Theorem 1]{ding1}(see also \cite[Chapter 3]{lu} for the weighted version).
  \item In 2002, Ding and Lu considered the extreme case $p=n/{\alpha}$, and proved the following result (see \cite[Theorem 1]{ding2}). Let $0<\alpha<n$ and $p=n/{\alpha}$. If $\Omega$ satisfies the $L^s$-Dini smoothness condition $\mathfrak{D}_{s}$ with $s>n/{(n-\alpha)}$(see Definition \ref{defd} below) and $\omega^{s'}\in A(p/{s'},\infty)$, then the operator $T_{\Omega,\alpha}$ is bounded from $L^p(\omega^p)$ to $\mathrm{BMO}_{\omega}$.
\end{enumerate}

On the other hand, the concept of classical Morrey space was originally introduced in \cite{morrey} to study the local regularity of solutions to second order elliptic partial differential equations. Nowadays the classical Morrey space has been studied extensively in the literature, and found a wide range of applications in harmonic analysis, potential theory and PDEs. Let us first recall the definition of the classical Morrey space. Let $1\leq p<\infty$ and $0\leq\kappa\leq1$. The classical Morrey space $\mathcal{M}^{p,\kappa}(\mathbb R^n)$ is defined to be the set of all locally integrable functions $f$ on $\mathbb R^n$ such that
\begin{equation*}
\begin{split}
\|f\|_{\mathcal{M}^{p,\kappa}}:=&\sup_{(x_0,r)\in\mathbb R^n\times(0,\infty)}
\bigg(\frac{1}{m(B(x_0,r))^{\kappa}}\int_{B(x_0,r)}|f(x)|^p\,dx\bigg)^{1/p}\\
=&\sup_{(x_0,r)\in\mathbb R^n\times(0,\infty)}\frac{1}{m(B(x_0,r))^{\kappa/p}}\big\|f\cdot\chi_{B(x_0,r)}\big\|_{L^p}<+\infty.
\end{split}
\end{equation*}
Here $B(x_0,r)$ is the Euclidean open ball with center $x\in\mathbb R^n$ and radius $r\in(0,\infty)$. It is easy to check that $\mathcal{M}^{p,\kappa}(\mathbb R^n)$ is a Banach function space with respect to the norm $\|\cdot\|_{\mathcal{M}^{p,\kappa}}$. $\mathcal{M}^{p,\kappa}(\mathbb R^n)$ is an extension of $L^p(\mathbb R^n)$ in the sense that $\mathcal{M}^{p,0}(\mathbb R^n)=L^p(\mathbb R^n)$. Note that $\mathcal{M}^{p,1}(\mathbb R^n)=L^{\infty}(\mathbb R^n)$ by the Lebesgue
differentiation theorem. If $\kappa<0$ or $\kappa>1$, then $\mathcal{M}^{p,\kappa}(\mathbb R^n)=\Theta$, where $\Theta$ is the set of all functions equivalent to $0$ on $\mathbb R^n$. It is very natural to consider the weighted version of Morrey spaces. In 2009, Komori and Shirai \cite{komori} introduced the weighted Morrey space $\mathcal{M}^{p,\kappa}(\omega,\nu)$, and studied the weighted boundedness of classical
operators in harmonic analysis, such as the Hardy--Littlewood maximal operator, the Calder\'{o}n--Zygmund operator and the fractional integral operator on this space. Let $1\leq p<\infty$, $0<\kappa<1$ and $\omega,\nu$ be two weight functions. The weighted Morrey space $\mathcal{M}^{p,\kappa}(\omega,\nu)$ is defined as the set of all integrable functions $f$ on $\mathbb R^n$ such that
\begin{equation*}
\begin{split}
\big\|f\big\|_{L^{p,\kappa}(\omega)}&:=\sup_{(x_0,r)\in\mathbb R^n\times(0,\infty)}\bigg(\frac{1}{\nu(B(x_0,r))^{\kappa}}\int_{B(x_0,r)}|f(x)|^p\omega(x)\,dx\bigg)^{1/p}\\
&=\sup_{(x_0,r)\in\mathbb R^n\times(0,\infty)}\frac{1}{\nu(B(x_0,r))^{\kappa/p}}\big\|f\cdot\chi_{B(x_0,r)}\big\|_{L^p(\omega)}<+\infty.
\end{split}
\end{equation*}
When $\omega=\nu$, we simply write $\mathcal{M}^{p,\kappa}(\omega,\nu)$ as $\mathcal{M}^{p,\kappa}(\omega)$. When $\omega\equiv1$, $\mathcal{M}^{p,\kappa}(\omega)=\mathcal{M}^{p,\kappa}(\mathbb R^n)$.
\begin{enumerate}
  \item In 2009, Komori and Shirai investigated the boundedness of $I_{\alpha}$ on weighted Morrey spaces, and showed that the Riesz potential operator $I_{\alpha}$ is bounded from $\mathcal{M}^{p,\kappa}(\omega^p,\omega^q)$ to $\mathcal{M}^{q,{\kappa q}/p}(\omega^q)$, whenever
      \begin{equation*}
      0<\alpha<n,1<p<n/{\alpha},1/q=1/p-{\alpha}/n,0<\kappa<p/q, ~\mbox{and} ~~\omega\in A(p,q),
      \end{equation*}
      see \cite[Theorem 3.6]{komori}. For the unweighted case, see also \cite{adams}, \cite{adams1} and \cite{wang3}.
  \item In 2013, the author established the corresponding results for $T_{\Omega,\alpha}$, when $0<\alpha<n$, $\Omega$ is homogeneous of degree zero on $\mathbb R^n$ and $\Omega\in L^s(\mathbb{S}^{n-1})$ with $1<s\leq\infty$. It is shown that the homogeneous fractional integral operator $T_{\Omega,\alpha}$ is bounded from $\mathcal{M}^{p,\kappa}(\omega^p,\omega^q)$ to $\mathcal{M}^{q,{\kappa q}/p}(\omega^q)$, whenever
  \begin{equation*}
  1\leq s'<p<n/{\alpha},1/q=1/p-{\alpha}/n,0<\kappa<p/q,~\mbox{and} ~\omega^{s'}\in A(p/{s'},q/{s'}),
  \end{equation*}
  see \cite[Theorem 1.2]{wang}(see also \cite{lu2} for the unweighted case).
  \item What do we have for the extreme case $\kappa=p/q=1-{(\alpha p)}/n$?Recently, the author obtained the following result. Let $0<\alpha<n$, $1\leq p<n/{\alpha}$ and $1/q=1/p-{\alpha}/n$. If $\Omega$ satisfies the $L^s$-Dini smoothness condition $\mathfrak{D}_{s}$ with $1\leq s'\leq p$ and $\omega^{s'}\in A(p/{s'},q/{s'})$, then the homogeneous fractional integral operator $T_{\Omega,\alpha}$ is bounded from $\mathcal{M}^{p,\kappa}(\omega^p,\omega^q)$ to $\mathrm{BMO}(\mathbb R^n)$ (see \cite[Theorem 2.2]{wang2}). Our result is new even for the case $\Omega\equiv1$ and $s=\infty$. In particular, if $1\leq p<n/{\alpha}$, $1/q=1/p-{\alpha}/n$ and $\omega\in A(p,q)$, then the Riesz potential operator $I_{\alpha}$ is bounded from $\mathcal{M}^{p,\kappa}(\omega^p,\omega^q)$ to $\mathrm{BMO}(\mathbb R^n)$ (see \cite[Corollary 2.3]{wang2}). It should be pointed out that the unweighted version of this result was also given in \cite{meng}.
\end{enumerate}
Motivated by these results, it is natural to ask what happens when $n/{\alpha}<p<\infty$(for weighted Lebesgue spaces) and $p/q<\kappa<1$(for weighted Morrey spaces). The main purpose of this paper is to answer the question above. Let us now introduce a class of Campanato spaces.

\begin{defin}
Let $1\leq\ell<\infty$ and $0\leq\gamma\leq1$. Then the Campanato space $\mathcal{C}^{\gamma,\ell}_{\omega}$ with respect to $\omega$ is defined to be the set of all integrable functions $\hbar$ on $\mathbb R^n$ such that
\begin{equation*}
\|\hbar\|_{\mathcal{C}^{\gamma,\ell}_{\omega}}:=\sup_{B\subset\mathbb R^n}\Big[\underset{y\in{B}}{\mathrm{ess\,sup}}\,\omega(y)\Big]
\frac{1}{m(B)^{\gamma/n}}\bigg(\frac{1}{m(B)}\int_{B}|\hbar(x)-\hbar_B|^{\ell}dx\bigg)^{1/{\ell}}<+\infty,
\end{equation*}
where the supremum is taken over all balls $B$ in $\mathbb R^n$.
\end{defin}

\begin{defin}
Let $1\leq\ell<\infty$, $0\leq\gamma\leq1$ and $\nu$ be a weight function. Then the weighted Campanato space $\mathcal{C}^{\gamma,\ell}({\nu})$ is defined as the set of all integrable functions $\hbar$ on $\mathbb R^n$ such that
\begin{equation*}
\|\hbar\|_{\mathcal{C}^{\gamma,\ell}({\nu})}:=\sup_{B\subset\mathbb R^n}\frac{1}{\nu(B)^{\gamma/n}}
\bigg(\frac{1}{m(B)}\int_{B}|\hbar(x)-\hbar_B|^{\ell}dx\bigg)^{1/{\ell}}<+\infty,
\end{equation*}
where the supremum is taken over all balls $B$ in $\mathbb R^n$.
\end{defin}
If $\gamma=0$ and $\ell=1$, then $\mathcal{C}^{\gamma,\ell}_{\omega}$ and $\mathcal{C}^{\gamma,\ell}({\nu})$ coincide with the space $\mathrm{BMO}_{\omega}$ and the space $\mathrm{BMO}(\mathbb R^n)$, respectively. Motivated by \cite{ding2} and \cite{wang2}, under the hypothesis
that $\Omega$ satisfies certain Dini-type smoothness condition, we will prove that the homogeneous fractional integral operator $T_{\Omega,\alpha}$
is bounded from $L^{p}(\omega^p)$ to $\mathcal{C}^{\gamma,\ell}_{\omega}$ for appropriate indices $p,\gamma$ and $\ell$, when $n/{\alpha}<p<\infty$. Moreover, it can also be shown that $T_{\Omega,\alpha}$ is bounded from $\mathcal{M}^{p,\kappa}(\omega^p,\omega^q)$ to $\mathcal{C}^{\gamma,\ell}(\omega^q)$ for appropriate indices $p,q,\gamma$ and $\ell$, when $p/q<\kappa<1$.

In this article, $C>0$ denotes a universal constant which is independent of the main parameters involved but whose value may be different from line to line. The symbol $\mathbf{X}\lesssim \mathbf{Y}$ means that there exists a positive constant $C$ such that $\mathbf{X}\leq C\mathbf{Y}$. We use the notation $p'=p/{(p-1)}$ and $1'=\infty$, $\infty'=1$.

\section{Main results}
\label{sec2}
In this section, we will establish boundedness properties of homogeneous fractional integral operators $T_{\Omega,\alpha}$ acting on weighted Lebesgue and Morrey spaces. Before stating our main theorems, let us now introduce some notations.
\begin{defin}\label{defd}
Let $1\leq s<\infty$ and $0\leq\beta\leq1$. We say that $\Omega$ satisfies the $L^{s;\beta}$-Dini smoothness condition $\mathfrak{D}_{s;\beta}$, if $\Omega\in L^s(\mathbb{S}^{n-1})$ is homogeneous of degree zero on $\mathbb R^n$, and
\begin{equation*}
\int_0^1\frac{\omega_s(\delta)}{\delta^{1+\beta}}\,d\delta<+\infty,
\end{equation*}
where $\omega_s(\delta)$ denotes the integral modulus of continuity of order $s$, which is defined by
\begin{equation}\label{964}
\omega_s(\delta):=\sup_{|\rho|<\delta}\bigg(\int_{\mathbb{S}^{n-1}}\big|\Omega(\rho x')-\Omega(x')\big|^sd\sigma(x')\bigg)^{1/s},
\end{equation}
and $\rho$ is a rotation in $\mathbb R^n$ and
\begin{equation*}
|\rho|:=\|\rho-I\|=\sup_{x'\in\mathbb{S}^{n-1}}\big|\rho x'-x'\big|.
\end{equation*}
We also say that $\Omega$ satisfies the $L^{\infty;\beta}$-Dini smoothness condition $\mathfrak{D}_{\infty;\beta}$, if $\Omega\in L^{\infty}(\mathbb{S}^{n-1})$ is homogeneous of degree zero on $\mathbb R^n$, and
\begin{equation*}
\int_0^1\frac{\omega_{\infty}(\delta)}{\delta^{1+\beta}}\,d\delta<+\infty,
\end{equation*}
where $\omega_{\infty}(\delta)$ is defined by
\begin{equation}\label{965}
\omega_{\infty}(\delta):=\sup_{|\rho|<\delta,x'\in \mathbb{S}^{n-1}}\big|\Omega(\rho x')-\Omega(x')\big|.
\end{equation}
Here $x'=x/{|x|}$ for any $x\in \mathbb R^n\setminus\{0\}$. For $\beta=0$, we simply write $\mathfrak{D}_{s;\beta}$ as $\mathfrak{D}_{s}$.
\end{defin}

\begin{rem}\label{remark10}
Let $1\leq s\leq\infty$. Since
\begin{equation*}
\int_0^1\frac{\omega_s(\delta)}{\delta}\,d\delta<\int_0^1\frac{\omega_s(\delta)}{\delta^{1+\beta}}\,d\delta
\end{equation*}
holds for any $\beta>0$, we know that if $\Omega$ satisfies the smoothness condition $\mathfrak{D}_{s;\beta}$, then it must satisfy the smoothness condition $\mathfrak{D}_{s}$.
\end{rem}

Our results can be stated as follows:
\begin{thm}\label{thm1}
Suppose that $\Omega$ satisfies the $L^{s;\beta}$-Dini smoothness condition $\mathfrak{D}_{s;\beta}$ with $n/{(n-\alpha)}\leq s\leq\infty$ and $0<\beta\leq1$. If $0<\alpha<1+n/p$, $n/{\alpha}<p<\infty$ and $\omega^{s'}\in A(p/{s'},\infty)$, then the operator $T_{\Omega,\alpha}$ is bounded from $L^{p}(\omega^p)$ to $\mathcal{C}^{\gamma,\ell}_{\omega}$ with $1\leq\ell<n/{(n-\alpha)}$, $\gamma:=n(\alpha/n-1/p)=\alpha-n/p$ and $\gamma<\beta$.
\end{thm}

\begin{rem}
When $n/{(n-\alpha)}\leq s\leq\infty$ and $n/{\alpha}<p<\infty$, one has
\begin{equation*}
s'\leq\frac{\,n\,}{\alpha}<p.
\end{equation*}
For the critical case $p=n/{\alpha}$ and $\gamma=0$, $\mathcal{C}^{\gamma,\ell}_{\omega}$ reduces to the space $\mathrm{BMO}_{\omega}$, then we have that $T_{\Omega,\alpha}$ is bounded from $L^{p}(\omega^p)$ to $\mathrm{BMO}_{\omega}$, which was given in \cite{ding2}.
\end{rem}

\begin{thm}\label{thm2}
Suppose that $\Omega$ satisfies the $L^{s;\beta}$-Dini smoothness condition $\mathfrak{D}_{s;\beta}$ with $n/{(n-\alpha)}<s\le\infty$ and $0<\beta\leq1$. If $0<\alpha<1+n/p$, $1\le s'\leq p<n/{\alpha}$, $1/q=1/p-{\alpha}/n$, $p/q<\kappa<1$ and $\omega^{s'}\in A(p/{s'},q/{s'})$, then the operator $T_{\Omega,\alpha}$ is bounded from $\mathcal{M}^{p,\kappa}(\omega^p,\omega^q)$ to $\mathcal{C}^{\gamma_*,\ell}(\omega^q)$ with $1\leq\ell<n/{(n-\alpha)}$, $\gamma_*:=n(\kappa/p-1/q)=\alpha-{(1-\kappa)n}/p$ and $\tau\cdot\gamma_*<\beta$, where
\begin{equation*}
\tau=
\begin{cases}
1+\frac{q/{s'}}{(p/{s'})'}, &\mbox{if}~ s'<p<n/{\alpha};\\
1,&\mbox{if}~ s'=p<n/{\alpha}.
\end{cases}
\end{equation*}
\end{thm}

\begin{rem}
When $1\le s'\leq p<n/{\alpha}$ and $1/q=1/p-{\alpha}/n$, one has
\begin{equation*}
\frac{p}{s'}\geq1 \quad \& \quad \frac{q}{s'}>1.
\end{equation*}
For the critical case $\kappa=p/q=1-{(\alpha p)}/n$ and $\gamma_*=0$, $\mathcal{C}^{\gamma_*,\ell}(\omega^q)$ reduces to the space $\mathrm{BMO}(\mathbb R^n)$, then we have that $T_{\Omega,\alpha}$ is bounded from $\mathcal{M}^{p,\kappa}(\omega^p,\omega^q)$ to $\mathrm{BMO}(\mathbb R^n)$, which was obtained in \cite{wang2}.
\end{rem}

\section{Some lemmas}
To prove our main theorems, we need some preliminary lemmas.

We first establish the following estimate, which will be often used in the sequel.
\begin{lem}\label{lem2}
Let $\Omega$ be homogeneous of degree zero and belong to $L^s(\mathbb{S}^{n-1})$ for certain $s\in(1,\infty]$, $0<\alpha<n$ and $1\leq\ell<n/{(n-\alpha)}\leq s$. Then for any $\mathcal{R}>0$, there is a constant $C>0$ independent of $\mathcal{R}$ such that
\begin{equation}\label{omega88}
\bigg(\int_{|x-y|<\mathcal{R}}\frac{|\Omega(x-y)|^{\ell}}{|x-y|^{(n-\alpha)\ell}}dx\bigg)^{1/{\ell}}
\leq C\|\Omega\|_{L^s(\mathbb{S}^{n-1})}\mathcal{R}^{n/{\ell}-(n-\alpha)}.
\end{equation}
\end{lem}

\begin{proof}
Since $\Omega\in L^s(\mathbb{S}^{n-1})$ with $s\geq n/{(n-\alpha)}>\ell$, by H\"{o}lder's inequality, we see that $\Omega\in L^{\ell}(\mathbb{S}^{n-1})$, and
\begin{equation*}
\|\Omega\|_{L^{\ell}(\mathbb{S}^{n-1})}\leq C\|\Omega\|_{L^s(\mathbb{S}^{n-1})}.
\end{equation*}
Note that $n-(n-\alpha)\ell>0$. Using polar coordinates, we can deduce that for any $\mathcal{R}>0$,

\begin{equation*}
\begin{split}
&\bigg(\int_{|x-y|<\mathcal{R}}\frac{|\Omega(x-y)|^{\ell}}{|x-y|^{(n-\alpha)\ell}}dx\bigg)^{1/{\ell}}\\
&=\bigg(\int_{0}^{\mathcal{R}}
\int_{\mathbb{S}^{n-1}}\frac{|\Omega(x')|^{\ell}}{\varrho^{(n-\alpha)\ell}}\varrho^{n-1}d\sigma(x')d\varrho\bigg)^{1/{\ell}}\\
&=\Big[\frac{1}{n-(n-\alpha)\ell}\Big]^{1/{\ell}}\cdot \mathcal{R}^{n/{\ell}-(n-\alpha)}\|\Omega\|_{L^{\ell}(\mathbb{S}^{n-1})}\\
&\leq C\cdot \mathcal{R}^{n/{\ell}-(n-\alpha)}\|\Omega\|_{L^s(\mathbb{S}^{n-1})}.
\end{split}
\end{equation*}
This proves \eqref{omega88}. Here $x'=x/{|x|}$ for any $0\neq x\in \mathbb R^n$.
\end{proof}

We need the following estimate which can be found in \cite[Lemma 1]{ding2}.

\begin{lem}\label{lem1}
Let $0<\alpha<n$. Suppose that $\Omega$ satisfies the $L^s$-Dini smoothness condition $\mathfrak{D}_{s}$ with $1<s\le\infty$. If there exists a constant $0<\vartheta<1/2$ such that if $|x|<\vartheta \mathcal{R}$, then we have
\begin{equation*}
\begin{split}
&\bigg(\int_{\mathcal{R}\leq|z|<2\mathcal{R}}
\bigg|\frac{\Omega(z-x)}{|z-x|^{n-\alpha}}-\frac{\Omega(z)}{|z|^{n-\alpha}}\bigg|^sdz\bigg)^{1/s}\\
&\leq C\cdot \mathcal{R}^{n/s-(n-\alpha)}\bigg(\frac{|x|}{\mathcal{R}}
+\int_{|x|/{2\mathcal{R}}}^{|x|/\mathcal{R}}\frac{\omega_{s}(\delta)}{\delta}d\delta\bigg),
\end{split}
\end{equation*}
where the constant $C>0$ is independent of $\mathcal{R}$ and $x$.
\end{lem}

Based on Lemma \ref{lem1}, we give the following estimate.
\begin{lem}\label{lem3}
Let $0<\alpha<n$. Suppose that $\Omega$ satisfies the $L^{s;\beta}$-Dini smoothness condition $\mathfrak{D}_{s;\beta}$ with $1<s\le\infty$ and $0<\beta\leq1$. Then for any ball $\mathcal{B}=B(x_0,r)$ and for any $x,y\in \mathcal{B}$,
\begin{equation}\label{wewant0}
\begin{split}
&\bigg(\int_{2^{j+1}\mathcal{B}\backslash 2^j\mathcal{B}}
\left|\frac{\Omega(x-z)}{|x-z|^{n-\alpha}}-\frac{\Omega(y-z)}{|y-z|^{n-\alpha}}\right|^sdz\bigg)^{1/s}\\
&\leq C\cdot\big(2^jr\big)^{n/s-(n-\alpha)}
\bigg[\frac{1}{2^{j-1}}+\frac{1}{2^{j\beta}}
\int_{|x-x_0|/{(2^{j+1}r)}}^{|x-x_0|/{(2^jr)}}\frac{\omega_s(\delta)}{\delta^{1+\beta}}d\delta
+\frac{1}{2^{j\beta}}
\int_{|y-x_0|/{(2^{j+1}r)}}^{|y-x_0|/{(2^jr)}}\frac{\omega_s(\delta)}{\delta^{1+\beta}}d\delta\bigg].
\end{split}
\end{equation}
Here $2\leq j\in \mathbb{N}$ and $\omega_s$ is as in \eqref{964} or \eqref{965}.
\end{lem}
\begin{proof}
For any ball $\mathcal{B}=B(x_0,r)$ and $0<\alpha<n$, since
\begin{equation*}
\begin{split}
\left|\frac{\Omega(x-z)}{|x-z|^{n-\alpha}}-\frac{\Omega(y-z)}{|y-z|^{n-\alpha}}\right|
&\leq\left|\frac{\Omega(x-z)}{|x-z|^{n-\alpha}}-\frac{\Omega(x_0-z)}{|x_0-z|^{n-\alpha}}\right|\\
&+\left|\frac{\Omega(x_0-z)}{|x_0-z|^{n-\alpha}}-\frac{\Omega(y-z)}{|y-z|^{n-\alpha}}\right|,
\end{split}
\end{equation*}
it then follows from Minkowski's inequality that for each $j\geq2$,
\begin{equation*}
\begin{split}
&\bigg(\int_{2^{j+1}\mathcal{B}\backslash 2^j\mathcal{B}}
\left|\frac{\Omega(x-z)}{|x-z|^{n-\alpha}}-\frac{\Omega(y-z)}{|y-z|^{n-\alpha}}\right|^sdz\bigg)^{1/s}\\
\leq&
\bigg(\int_{2^{j+1}\mathcal{B}\backslash 2^j\mathcal{B}}\left|\frac{\Omega(x-z)}{|x-z|^{n-\alpha}}-\frac{\Omega(x_0-z)}{|x_0-z|^{n-\alpha}}\right|^sdz\bigg)^{1/s}\\
&+\bigg(\int_{2^{j+1}\mathcal{B}\backslash 2^j\mathcal{B}}
\left|\frac{\Omega(x_0-z)}{|x_0-z|^{n-\alpha}}-\frac{\Omega(y-z)}{|y-z|^{n-\alpha}}\right|^sdz\bigg)^{1/s}.
\end{split}
\end{equation*}
For any given $x\in\mathcal{B}$ and $z\in 2^{j+1}\mathcal{B}\backslash 2^j\mathcal{B}$ with $j\geq2$, it is easy to verify that
\begin{equation*}
|x-x_0|<\frac{\,1\,}{4}|z-x_0|.
\end{equation*}
As pointed out in Remark \ref{remark10}, for any $0<\beta\leq 1$ and $1<s\le\infty$, the condition $\mathfrak{D}_{s;\beta}$ is stronger than the $L^s$-Dini condition $\mathfrak{D}_{s}$. According to Lemma \ref{lem1}(here we take $\mathcal{R}=2^jr$), we get
\begin{equation*}
\begin{split}
&\bigg(\int_{2^{j+1}\mathcal{B}\backslash2^j\mathcal{B}}
\left|\frac{\Omega(x-z)}{|x-z|^{n-\alpha}}-\frac{\Omega(x_0-z)}{|x_0-z|^{n-\alpha}}\right|^sdz\bigg)^{1/s}\\
&=\bigg(\int_{2^jr\leq|z-x_0|<2^{j+1}r}\left|\frac{\Omega(z-x_0-(x-x_0))}{|z-x_0-(x-x_0)|^{n-\alpha}}
-\frac{\Omega(z-x_0)}{|z-x_0|^{n-\alpha}}\right|^sdz\bigg)^{1/s}\\
&\leq C\cdot\big(2^jr\big)^{n/s-(n-\alpha)}\bigg[\frac{|x-x_0|}{2^jr}+\int_{|x-x_0|/{(2^{j+1}r)}}^{|x-x_0|/{(2^jr)}}
\frac{\omega_s(\delta)}{\delta}d\delta\bigg].
\end{split}
\end{equation*}
Observe that for any $0<\beta\leq1$,
\begin{equation*}
\int_{|x-x_0|/{(2^{j+1}r)}}^{|x-x_0|/{(2^jr)}}\frac{\omega_s(\delta)}{\delta}d\delta
\leq\Big(\frac{|x-x_0|}{2^{j}r}\Big)^{\beta}
\int_{|x-x_0|/{(2^{j+1}r)}}^{|x-x_0|/{(2^jr)}}\frac{\omega_s(\delta)}{\delta^{1+\beta}}d\delta.
\end{equation*}
Hence, for any given $x\in \mathcal{B}=B(x_0,r)$,
\begin{equation}\label{wewant1}
\begin{split}
&\bigg(\int_{2^{j+1}\mathcal{B}\backslash2^j\mathcal{B}}
\left|\frac{\Omega(x-z)}{|x-z|^{n-\alpha}}-\frac{\Omega(x_0-z)}{|x_0-z|^{n-\alpha}}\right|^sdz\bigg)^{1/s}\\
&\leq C\cdot\big(2^jr\big)^{n/s-(n-\alpha)}\bigg[\frac{|x-x_0|}{2^jr}+\Big(\frac{|x-x_0|}{2^{j}r}\Big)^{\beta}
\int_{|x-x_0|/{(2^{j+1}r)}}^{|x-x_0|/{(2^jr)}}\frac{\omega_s(\delta)}{\delta^{1+\beta}}d\delta\bigg]\\
&\leq C\cdot\big(2^jr\big)^{n/s-(n-\alpha)}\bigg[\frac{1}{2^j}+\frac{1}{2^{j\beta}}
\int_{|x-x_0|/{(2^{j+1}r)}}^{|x-x_0|/{(2^jr)}}\frac{\omega_s(\delta)}{\delta^{1+\beta}}d\delta\bigg].
\end{split}
\end{equation}
Similarly, for any given $y\in \mathcal{B}=B(x_0,r)$, we can also obtain
\begin{equation}\label{wewant2}
\begin{split}
&\bigg(\int_{2^{j+1}\mathcal{B}\backslash 2^j\mathcal{B}}
\left|\frac{\Omega(x_0-z)}{|x_0-z|^{n-\alpha}}-\frac{\Omega(y-z)}{|y-z|^{n-\alpha}}\right|^sdz\bigg)^{1/s}\\
&\leq C\cdot\big(2^jr\big)^{n/s-(n-\alpha)}
\bigg[\frac{|y-x_0|}{2^jr}+\Big(\frac{|y-x_0|}{2^{j}r}\Big)^{\beta}
\int_{|y-x_0|/{(2^{j+1}r)}}^{|y-x_0|/{(2^jr)}}\frac{\omega_s(\delta)}{\delta^{1+\beta}}d\delta\bigg]\\
&\leq C\cdot\big(2^jr\big)^{n/s-(n-\alpha)}
\bigg[\frac{1}{2^j}+\frac{1}{2^{j\beta}}
\int_{|y-x_0|/{(2^{j+1}r)}}^{|y-x_0|/{(2^jr)}}\frac{\omega_s(\delta)}{\delta^{1+\beta}}d\delta\bigg].
\end{split}
\end{equation}
Combining the above estimates \eqref{wewant1} and \eqref{wewant2} leads to the estimate \eqref{wewant0}, which completes the proof of Lemma \ref{lem3}.
\end{proof}

Finally, we also need the following known results about the class of $A_p$ and the class of $A(p,q)$,
which can be found in \cite[Chapter 7]{grafakos} and \cite[Chapter 3]{lu}.
\begin{lem}\label{lem4}
Let $1\leq p<\infty$ and $\omega\in A_p$. Then we have
\begin{enumerate}
  \item $\omega(x)\,dx$ is a doubling measure. To be more precise, for any ball $\mathcal{B}$ in $\mathbb R^n$,
\begin{equation*}
\omega(2\mathcal{B})\leq C\cdot\omega(\mathcal{B}),
\end{equation*}
where the constant $C>0$ is independent of $\mathcal{B}$.
  \item For any $\lambda>1$ and for any ball $\mathcal{B}$ in $\mathbb R^n$,
\begin{equation*}
\omega(\lambda\mathcal{B})\leq \lambda^{np}[\omega]_{A_p}\cdot\omega(\mathcal{B}).
\end{equation*}
\end{enumerate}
\end{lem}

\begin{lem}\label{lem5}
Let $0<\alpha<n$, $1\leq p<n/{\alpha}$ and $1/q=1/p-{\alpha}/n$. Then we have
\begin{enumerate}
  \item If $p>1$, then
\begin{equation*}
\omega\in A(p,q)\Longleftrightarrow \omega^q\in A_{\tau},
\end{equation*}
where $\tau=1+q/{p'}$.
  \item If $p=1$, then
\begin{equation*}
\omega\in A(1,q)\Longleftrightarrow \omega^q\in A_1.
\end{equation*}
\end{enumerate}
\end{lem}

\section{Proof of Theorem $\ref{thm1}$}
\begin{proof}[Proof of Theorem $\ref{thm1}$]
Let $f\in L^{p}(\omega^p)$ with $s'<p<\infty$ and $\omega^{s'}\in A(p/{s'},\infty)$. By definition, it suffices to verify that for any ball $\mathcal{B}=B(x_0,r)$ with $(x_0,r)\in\mathbb R^n\times(0,\infty)$ and $f\in L^{p}(\omega^p)$, the following estimate holds:
\begin{equation}\label{mainw}
\Big[\underset{y\in \mathcal{B}}{\mathrm{ess\,sup}}\,\omega(y)\Big]\cdot\frac{1}{m(\mathcal{B})^{\gamma/n}}
\bigg(\frac{1}{m(\mathcal{B})}\int_{\mathcal{B}}\big|T_{\Omega,\alpha}f(x)-(T_{\Omega,\alpha}f)_{\mathcal{B}}\big|^{\ell}dx\bigg)^{1/{\ell}}
\lesssim\|f\|_{L^{p}(\omega^p)}.
\end{equation}
For this purpose, we decompose $f$ as
\begin{equation*}
f(x)=f(x)\cdot\chi_{4\mathcal{B}}(x)+f(x)\cdot\chi_{(4\mathcal{B})^{\complement}}(x):=f_1(x)+f_2(x),
\end{equation*}
where $4\mathcal{B}=B(x_0,4r)$. By the linearity of $T_{\Omega,\alpha}$, we have
\begin{equation*}
T_{\Omega,\alpha}f(x)=T_{\Omega,\alpha}f_1(x)+T_{\Omega,\alpha}f_2(x)\quad \& \quad
(T_{\Omega,\alpha}f)_{\mathcal{B}}=(T_{\Omega,\alpha}f_1)_{\mathcal{B}}+(T_{\Omega,\alpha}f_2)_{\mathcal{B}}.
\end{equation*}
Then for $\ell\geq1$, by Minkowski's inequality, we can write
\begin{equation*}
\begin{split}
&\Big[\underset{y\in \mathcal{B}}{\mathrm{ess\,sup}}\,\omega(y)\Big]\cdot
\frac{1}{m(\mathcal{B})^{\gamma/n}}
\bigg(\frac{1}{m(\mathcal{B})}\int_{\mathcal{B}}\big|T_{\Omega,\alpha}f(x)-(T_{\Omega,\alpha}f)_{\mathcal{B}}\big|^{\ell}dx\bigg)^{1/{\ell}}\\
&\leq\Big[\underset{y\in \mathcal{B}}{\mathrm{ess\,sup}}\,\omega(y)\Big]\cdot
\frac{1}{m(\mathcal{B})^{\gamma/n+1/{\ell}}}
\bigg(\int_{\mathcal{B}}\big|T_{\Omega,\alpha}f_1(x)-(T_{\Omega,\alpha}f_1)_{\mathcal{B}}\big|^{\ell}\,dx\bigg)^{1/{\ell}}\\
&+\Big[\underset{y\in \mathcal{B}}{\mathrm{ess\,sup}}\,\omega(y)\Big]\cdot
\frac{1}{m(\mathcal{B})^{\gamma/n+1/{\ell}}}
\bigg(\int_{\mathcal{B}}\big|T_{\Omega,\alpha}f_2(x)-(T_{\Omega,\alpha}f_2)_{\mathcal{B}}\big|^{\ell}\,dx\bigg)^{1/{\ell}}\\
&:=\mathrm{I+II}.
\end{split}
\end{equation*}
Let us first consider the term $\mathrm{I}$. For $\ell\geq1$, by using H\"{o}lder's inequality, we have
\begin{equation}\label{doub}
\begin{split}
\big|(T_{\Omega,\alpha}f_1)_{\mathcal{B}}\big|
&\leq\frac{1}{m(\mathcal{B})}\int_{\mathcal{B}}\big|T_{\Omega,\alpha}f_1(y)\big|\,dy\\
&\leq\bigg(\frac{1}{m(\mathcal{B})}\int_{\mathcal{B}}\big|T_{\Omega,\alpha}f_1(y)\big|^{\ell}\,dy\bigg)^{1/{\ell}}.
\end{split}
\end{equation}
Observe that when $x\in \mathcal{B}$ and $y\in 4\mathcal{B}$, we have $|x-y|<5r$. Applying Minkowski's integral inequality and \eqref{doub}, we can see that
\begin{equation*}
\begin{split}
\mathrm{I}&\leq\Big[\underset{y\in \mathcal{B}}{\mathrm{ess\,sup}}\,\omega(y)\Big]\cdot
\frac{2}{m(\mathcal{B})^{\gamma/n+1/{\ell}}}\bigg(\int_{\mathcal{B}}\big|T_{\Omega,\alpha}f_1(x)\big|^{\ell}\,dx\bigg)^{1/{\ell}}\\
&\leq\Big[\underset{y\in \mathcal{B}}{\mathrm{ess\,sup}}\,\omega(y)\Big]\cdot
\frac{2}{m(\mathcal{B})^{\gamma/n+1/{\ell}}}
\bigg(\int_{\mathcal{B}}\bigg[\int_{4\mathcal{B}}\frac{|\Omega(x-y)|}{|x-y|^{n-\alpha}}|f(y)|\,dy\bigg]^{\ell}dx\bigg)^{1/{\ell}}\\
&\leq\Big[\underset{y\in \mathcal{B}}{\mathrm{ess\,sup}}\,\omega(y)\Big]\cdot
\frac{2}{m(\mathcal{B})^{\gamma/n+1/{\ell}}}
\int_{4\mathcal{B}}|f(y)|\bigg(\int_{|x-y|<5r}\frac{|\Omega(x-y)|^{\ell}}{|x-y|^{(n-\alpha)\ell}}\,dx\bigg)^{1/{\ell}}dy.
\end{split}
\end{equation*}
Notice that $1\leq\ell<n/{(n-\alpha)}$. It follows directly from Lemma \ref{lem2} that for any $y\in 4\mathcal{B}$,
\begin{equation}\label{cal2}
\begin{split}
\bigg(\int_{|x-y|<5r}\frac{|\Omega(x-y)|^{\ell}}{|x-y|^{(n-\alpha)\ell}}\,dx\bigg)^{1/{\ell}}
&\leq C\cdot(5r)^{n/{\ell}-(n-\alpha)}\|\Omega\|_{L^s(\mathbb{S}^{n-1})}\\
&\lesssim m(\mathcal{B})^{1/{\ell}-(1-\alpha/n)}\|\Omega\|_{L^s(\mathbb{S}^{n-1})}.
\end{split}
\end{equation}
On the other hand, for $1\leq s'<p<\infty$, it follows from H\"older's inequality that
\begin{equation}\label{wang2}
\begin{split}
&\int_{4\mathcal{B}}|f(y)|\,dy=\int_{4\mathcal{B}}|f(y)|\omega(y)\cdot\omega(y)^{-1}\,dy\\
&\leq\bigg(\int_{4\mathcal{B}}|f(y)|^p\omega(y)^p\,dy\bigg)^{1/p}\cdot\bigg(\int_{4\mathcal{B}}\omega(y)^{-p'}dy\bigg)^{1/{p'}}.
\end{split}
\end{equation}
Notice that
\begin{equation}\label{cal3}
s'\cdot\big(p/{s'}\big)'=s'\cdot\frac{1}{1-\frac{s'}{p}}=\frac{ps'}{p-s'},
\end{equation}
and
\begin{equation}\label{cal4}
\frac{1}{s'\cdot(p/{s'})'}=\frac{1}{s'}\cdot\Big(1-\frac{1}{p/{s'}}\Big)=\frac{1}{s'}-\frac{\,1\,}{p}.
\end{equation}
By a simple calculation, we can see that when $s'\geq1$,
\begin{equation}\label{cal5}
\frac{ps'}{p-s'}\geq\frac{p}{p-1}=p'.
\end{equation}
Using H\"older's inequality, \eqref{cal3} and \eqref{cal5}, we get
\begin{equation}\label{wang3}
\begin{split}
\bigg(\frac{1}{m(4\mathcal{B})}\int_{4\mathcal{B}} \omega(y)^{-p'}dy\bigg)^{1/{p'}}
&\leq\bigg(\frac{1}{m(4\mathcal{B})}\int_{4\mathcal{B}} \omega(y)^{-s'(p/{s'})'}dy\bigg)^{\frac{1}{s'(p/{s'})'}}\\
&=\bigg(\frac{1}{m(4\mathcal{B})}\int_{4\mathcal{B}}\omega^{s'}(y)^{-(p/{s'})'}dy\bigg)^{\frac{1}{s'(p/{s'})'}}.
\end{split}
\end{equation}
Noting that
\begin{equation*}
\frac{\gamma}{\,n\,}=\frac{\alpha}{\,n\,}-\frac{1}{\,p\,}=\frac{1}{p'}+\frac{\alpha}{\,n\,}-1.
\end{equation*}
Combining the above estimates \eqref{cal2}, \eqref{wang2} and \eqref{wang3} leads to
\begin{equation*}
\begin{split}
\mathrm{I}&\lesssim\|\Omega\|_{L^s(\mathbb{S}^{n-1})}\cdot
\frac{m(4\mathcal{B})^{1/{p'}}}{m(\mathcal{B})^{\gamma/n+1-\alpha/n}}
\bigg(\int_{4\mathcal{B}}|f(y)|^p\omega(y)^{p}\,dy\bigg)^{1/p}\\
&\times\Big[\underset{y\in \mathcal{B}}{\mathrm{ess\,sup}}\,\omega(y)\Big]
\bigg(\frac{1}{m(4\mathcal{B})}\int_{4\mathcal{B}}\omega^{s'}(y)^{-(p/{s'})'}dy\bigg)^{\frac{1}{s'(p/{s'})'}}\\
&\lesssim\|\Omega\|_{L^s(\mathbb{S}^{n-1})}\|f\|_{L^p(\omega^p)}\\
&\times\Big[\underset{y\in 4\mathcal{B}}{\mathrm{ess\,sup}}\,\omega(y)\Big]
\bigg(\frac{1}{m(4\mathcal{B})}\int_{4\mathcal{B}}\omega^{s'}(y)^{-(p/{s'})'}dy\bigg)^{\frac{1}{s'(p/{s'})'}}\\
&=\|\Omega\|_{L^s(\mathbb{S}^{n-1})}\|f\|_{L^p(\omega^p)}\\
&\times\bigg[\Big[\underset{y\in 4\mathcal{B}}{\mathrm{ess\,sup}}\,\omega^{s'}(y)\Big]
\bigg(\frac{1}{m(4\mathcal{B})}\int_{4\mathcal{B}}\omega^{s'}(y)^{-(p/{s'})'}dy\bigg)^{\frac{1}{(p/{s'})'}}\bigg]^{\frac{1}{s'}}\\
&\leq C\|\Omega\|_{L^s(\mathbb{S}^{n-1})}\|f\|_{L^p(\omega^p)},
\end{split}
\end{equation*}
where in the last inequality we have invoked the condition $\omega^{s'}\in A(p/{s'},\infty)$. This gives the desired result.
Let us now turn to the estimate of the second term $\mathrm{II}$. It is easy to see that
\begin{equation*}
\begin{split}
\mathrm{II}&\leq\Big[\underset{y\in \mathcal{B}}{\mathrm{ess\,sup}}\,\omega(y)\Big]
\cdot\frac{1}{m(\mathcal{B})^{\gamma/n}}\\
&\times\bigg(\frac{1}{m(\mathcal{B})}\int_{\mathcal{B}}\bigg[\frac{1}{m(\mathcal{B})}
\int_{\mathcal{B}}\big|T_{\Omega,\alpha}f_2(x)-T_{\Omega,\alpha}f_2(y)\big|\,dy\bigg]^{\ell}dx\bigg)^{1/{\ell}}\\
\end{split}
\end{equation*}
\begin{equation*}
\begin{split}
&\leq\Big[\underset{y\in \mathcal{B}}{\mathrm{ess\,sup}}\,\omega(y)\Big]\cdot\frac{1}{m(\mathcal{B})^{\gamma/n}}\\
&\times\bigg(\frac{1}{m(\mathcal{B})}\int_{\mathcal{B}}
\bigg[\frac{1}{m(\mathcal{B})}\int_{\mathcal{B}}\bigg\{\sum_{j=2}^\infty\int_{2^{j+1}\mathcal{B}\backslash 2^j\mathcal{B}}
\left|\frac{\Omega(x-z)}{|x-z|^{n-\alpha}}-\frac{\Omega(y-z)}{|y-z|^{n-\alpha}}\right|\cdot|f(z)|\,dz\bigg\}dy\bigg]^{\ell}dx\bigg)^{1/{\ell}}.
\end{split}
\end{equation*}
Applying H\"{o}lder's inequality and Lemma \ref{lem3}, we obtain that for any $x,y\in \mathcal{B}$,
\begin{equation*}
\begin{split}
&\int_{2^{j+1}\mathcal{B}\backslash 2^j\mathcal{B}}
\left|\frac{\Omega(x-z)}{|x-z|^{n-\alpha}}-\frac{\Omega(y-z)}{|y-z|^{n-\alpha}}\right|\cdot|f(z)|\,dz\\
&\leq\bigg(\int_{2^{j+1}\mathcal{B}\backslash 2^j\mathcal{B}}
\left|\frac{\Omega(x-z)}{|x-z|^{n-\alpha}}-\frac{\Omega(y-z)}{|y-z|^{n-\alpha}}\right|^sdz\bigg)^{1/s}
\cdot\bigg(\int_{2^{j+1}\mathcal{B}\backslash 2^j\mathcal{B}}|f(z)|^{s'}dz\bigg)^{1/{s'}}\\
&\leq C\cdot\big(2^jr\big)^{n/s-(n-\alpha)}
\bigg[\frac{1}{2^{j-1}}+\frac{1}{2^{j\beta}}
\int_{|x-x_0|/{(2^{j+1}r)}}^{|x-x_0|/{(2^jr)}}\frac{\omega_s(\delta)}{\delta^{1+\beta}}d\delta
+\frac{1}{2^{j\beta}}
\int_{|y-x_0|/{(2^{j+1}r)}}^{|y-x_0|/{(2^jr)}}\frac{\omega_s(\delta)}{\delta^{1+\beta}}d\delta\bigg]\\
&\times\bigg(\int_{2^{j+1}\mathcal{B}}|f(z)|^{s'}dz\bigg)^{1/{s'}},
\end{split}
\end{equation*}
for each fixed integer $j\geq2$. On the other hand, it follows from H\"{o}lder's inequality with exponent $p/{s'}>1$ and the condition $\omega^{s'}\in A(p/{s'},\infty)$ that
\begin{equation*}
\begin{split}
&\bigg(\int_{2^{j+1}\mathcal{B}}|f(z)|^{s'}dz\bigg)^{1/{s'}}
=\bigg(\int_{2^{j+1}\mathcal{B}}|f(z)|^{s'}\omega(z)^{s'}\cdot \omega(z)^{-s'}dz\bigg)^{1/{s'}}\\
&\leq\bigg(\int_{2^{j+1}\mathcal{B}}|f(z)|^{p}\omega(z)^{p}dz\bigg)^{1/p}
\cdot\bigg(\int_{2^{j+1}\mathcal{B}}\omega^{s'}(z)^{-(p/{s'})'}dz\bigg)^{\frac{1}{s'(p/{s'})'}}\\
&\leq C\|f\|_{L^p(\omega^p)}m(2^{j+1}\mathcal{B})^{\frac{1}{s'(p/{s'})'}}
\cdot\Big[\underset{z\in 2^{j+1}\mathcal{B}}{\mathrm{ess\,sup}}\,\omega^{s'}(z)\Big]^{-\frac{1}{s'}}\\
&=C\|f\|_{L^p(\omega^p)}m(2^{j+1}\mathcal{B})^{\frac{1}{s'}-\frac{\,1\,}{p}}
\cdot\Big[\underset{z\in 2^{j+1}\mathcal{B}}{\mathrm{ess\,sup}}\,\omega(z)\Big]^{-1},
\end{split}
\end{equation*}
where in the last step we have used the equation \eqref{cal4}. Consequently, for any $x,y\in \mathcal{B}$, we deduce that
\begin{equation*}
\begin{split}
&\sum_{j=2}^\infty\int_{2^{j+1}\mathcal{B}\backslash 2^j\mathcal{B}}
\left|\frac{\Omega(x-z)}{|x-z|^{n-\alpha}}-\frac{\Omega(y-z)}{|y-z|^{n-\alpha}}\right|\cdot|f(z)|\,dz\\
&\lesssim\|f\|_{L^p(\omega^p)}\sum_{j=2}^\infty m(2^{j+1}\mathcal{B})^{\frac{\,1\,}{s}-(1-\frac{\alpha}{n})}\cdot m(2^{j+1}\mathcal{B})^{\frac{1}{s'}-\frac{\,1\,}{p}}
\cdot\Big[\underset{z\in 2^{j+1}\mathcal{B}}{\mathrm{ess\,sup}}\,\omega(z)\Big]^{-1}\\
&\times\bigg[\frac{1}{2^{j-1}}+\frac{1}{2^{j\beta}}
\int_{|x-x_0|/{(2^{j+1}r)}}^{|x-x_0|/{(2^jr)}}\frac{\omega_s(\delta)}{\delta^{1+\beta}}d\delta
+\frac{1}{2^{j\beta}}
\int_{|y-x_0|/{(2^{j+1}r)}}^{|y-x_0|/{(2^jr)}}\frac{\omega_s(\delta)}{\delta^{1+\beta}}d\delta\bigg]\\
&=\|f\|_{L^p(\omega^p)}\sum_{j=2}^\infty m(2^{j+1}\mathcal{B})^{{\gamma}/{n}}
\cdot\Big[\underset{z\in 2^{j+1}\mathcal{B}}{\mathrm{ess\,sup}}\,\omega(z)\Big]^{-1}\\
&\times\bigg[\frac{1}{2^{j-1}}+\frac{1}{2^{j\beta}}
\int_{|x-x_0|/{(2^{j+1}r)}}^{|x-x_0|/{(2^jr)}}\frac{\omega_s(\delta)}{\delta^{1+\beta}}d\delta
+\frac{1}{2^{j\beta}}
\int_{|y-x_0|/{(2^{j+1}r)}}^{|y-x_0|/{(2^jr)}}\frac{\omega_s(\delta)}{\delta^{1+\beta}}d\delta\bigg],
\end{split}
\end{equation*}
where in the last step we have used the fact that $\gamma/n=\alpha/n-1/p$. Therefore, we obtain
\begin{equation*}
\begin{split}
\mathrm{II}&\lesssim\|f\|_{L^p(\omega^p)}\Big[\underset{z\in \mathcal{B}}{\mathrm{ess\,sup}}\,\omega(z)\Big]
\sum_{j=2}^\infty\frac{m(2^{j+1}\mathcal{B})^{{\gamma}/{n}}}{m(\mathcal{B})^{\gamma/n}}
\cdot\Big[\underset{z\in 2^{j+1}\mathcal{B}}{\mathrm{ess\,sup}}\,\omega(z)\Big]^{-1}\\
&\times\bigg[\frac{1}{2^{j-1}}+\frac{1}{2^{j\beta}}
\int_{|x-x_0|/{(2^{j+1}r)}}^{|x-x_0|/{(2^jr)}}\frac{\omega_s(\delta)}{\delta^{1+\beta}}d\delta
+\frac{1}{2^{j\beta}}
\int_{|y-x_0|/{(2^{j+1}r)}}^{|y-x_0|/{(2^jr)}}\frac{\omega_s(\delta)}{\delta^{1+\beta}}d\delta\bigg]\\
&\lesssim\|f\|_{L^p(\omega^p)}\sum_{j=2}^\infty\bigg[\frac{1}{2^{j(1-\gamma)}}+\frac{1}{2^{j(\beta-\gamma)}}
\int_{0}^{1}\frac{\omega_s(\delta)}{\delta^{1+\beta}}d\delta
+\frac{1}{2^{j(\beta-\gamma)}}
\int_{0}^{1}\frac{\omega_s(\delta)}{\delta^{1+\beta}}d\delta\bigg]\\
&\leq C\|f\|_{L^p(\omega^p)}\bigg[1+2\int_0^1\frac{\omega_s(\delta)}{\delta^{1+\beta}}d\delta\bigg],
\end{split}
\end{equation*}
where the last inequality follows from the fact that $\gamma<\beta\leq1$. Combining the above estimates for $\mathrm{I}$ and $\mathrm{II}$ leads to the estimate \eqref{mainw}. By taking the supremum over all balls $\mathcal{B}\subset\mathbb R^n$, we conclude the proof of Theorem \ref{thm1}.
\end{proof}

\section{Proof of Theorem $\ref{thm2}$}
\begin{proof}[Proof of Theorem $\ref{thm2}$]
Let $f\in\mathcal{M}^{p,\kappa}(\omega^p,\omega^q)$ with $s'\leq p<n/{\alpha}$, $1/q=1/p-{\alpha}/n$ and $\omega^{s'}\in A(p/{s'},q/{s'})$. By definition, it suffices to prove that for any ball $\mathcal{B}=B(x_0,r)$ with $(x_0,r)\in\mathbb R^n\times(0,\infty)$ and $f\in\mathcal{M}^{p,\kappa}(\omega^p,\omega^q)$, the following estimate holds:
\begin{equation}\label{mainwh}
\frac{1}{\omega^q(\mathcal{B})^{{\gamma_*}/n}}
\bigg(\frac{1}{m(\mathcal{B})}\int_{\mathcal{B}}\big|T_{\Omega,\alpha}f(x)-(T_{\Omega,\alpha}f)_{\mathcal{B}}\big|^{\ell}\,dx\bigg)^{1/{\ell}}
\lesssim\|f\|_{\mathcal{M}^{p,\kappa}(\omega^p,\omega^q)}.
\end{equation}
To this end, as in the proof of Theorem \ref{thm1}, we again decompose $f$ as
\begin{equation*}
f(x)=f(x)\cdot\chi_{4\mathcal{B}}(x)+f(x)\cdot\chi_{(4\mathcal{B})^{\complement}}(x):=f_1(x)+f_2(x).
\end{equation*}
For $\ell\geq1$, by using Minkowski's inequality, we can write
\begin{equation*}
\begin{split}
&\frac{1}{\omega^q(\mathcal{B})^{{\gamma_*}/n}}
\bigg(\frac{1}{m(\mathcal{B})}\int_{\mathcal{B}}\big|T_{\Omega,\alpha}f(x)-(T_{\Omega,\alpha}f)_{\mathcal{B}}\big|^{\ell}\,dx\bigg)^{1/{\ell}}\\
&\leq \frac{1}{\omega^q(\mathcal{B})^{{\gamma_*}/n}}
\bigg(\frac{1}{m(\mathcal{B})}\int_{\mathcal{B}}\big|T_{\Omega,\alpha}f_1(x)-(T_{\Omega,\alpha}f_1)_{\mathcal{B}}\big|^{\ell}\,dx\bigg)^{1/{\ell}}\\
&+\frac{1}{\omega^q(\mathcal{B})^{{\gamma_*}/n}}\bigg(\frac{1}{m(\mathcal{B})}
\int_{\mathcal{B}}\big|T_{\Omega,\alpha}f_2(x)-(T_{\Omega,\alpha}f_2)_{\mathcal{B}}\big|^{\ell}\,dx\bigg)^{1/{\ell}}\\
&:=\mathrm{III+IV}.
\end{split}
\end{equation*}
First let us consider the term III. Applying the estimates \eqref{doub} and \eqref{cal2} together with Minkowski's integral inequality, we can deduce that
\begin{equation*}
\begin{split}
\mathrm{III}&\leq\frac{2}{\omega^q(\mathcal{B})^{{\gamma_*}/n}}
\bigg(\frac{1}{m(\mathcal{B})}\int_{\mathcal{B}}|T_{\Omega,\alpha}f_1(x)|^{\ell}\,dx\bigg)^{1/{\ell}}\\
&\leq\frac{2}{\omega^q(\mathcal{B})^{{\gamma_*}/n}}
\bigg(\frac{1}{m(\mathcal{B})}\int_{\mathcal{B}}\bigg[\int_{4\mathcal{B}}\frac{|\Omega(x-y)|}{|x-y|^{n-\alpha}}|f(y)|\,dy\bigg]^{\ell}dx\bigg)^{1/{\ell}}\\
&\leq\frac{2}{\omega^q(\mathcal{B})^{{\gamma_*}/n}}\cdot\frac{1}{m(\mathcal{B})^{1/{\ell}}}
\int_{4\mathcal{B}}|f(y)|\bigg(\int_{|x-y|<5r}\frac{|\Omega(x-y)|^{\ell}}{|x-y|^{(n-\alpha)\ell}}\,dx\bigg)^{1/{\ell}}dy\\
&\leq C\|\Omega\|_{L^s(\mathbb{S}^{n-1})}
\frac{1}{\omega^q(B)^{{\gamma_*}/n}}\cdot\frac{1}{m(\mathcal{B})^{1-\alpha/n}}\int_{4\mathcal{B}}|f(y)|\,dy.
\end{split}
\end{equation*}
On the other hand, by using H\"older's inequality, we obtain
\begin{equation}\label{cal22}
\begin{split}
&\int_{4\mathcal{B}}|f(y)|\,dy=\int_{4\mathcal{B}}|f(y)|\omega(y)\cdot\omega(y)^{-1}\,dy\\
&\leq\bigg(\int_{4\mathcal{B}}|f(y)|^p\omega(y)^p\,dy\bigg)^{1/p}
\cdot\bigg(\int_{4\mathcal{B}} \omega(y)^{-p'}dy\bigg)^{1/{p'}}.
\end{split}
\end{equation}
Notice that
\begin{equation}
s'\cdot\Big(\frac{p}{s'}\Big)'=\frac{s'}{1-\frac{s'}{p}}=\frac{ps'}{p-s'}\geq p'.
\end{equation}
This fact, together with H\"older's inequality and the condition $\omega^{s'}\in A(p/{s'},q/{s'})$, implies
\begin{equation}\label{cal23}
\begin{split}
\bigg(\frac{1}{m(4\mathcal{B})}\int_{4\mathcal{B}}\omega(y)^{-p'}dy\bigg)^{1/{p'}}
&\leq\bigg(\frac{1}{m(4\mathcal{B})}\int_{4\mathcal{B}} \omega(y)^{-s'(p/{s'})'}dy\bigg)^{\frac{1}{s'(p/{s'})'}}\\
&=\bigg(\frac{1}{m(4\mathcal{B})}\int_{4\mathcal{B}}\omega^{s'}(y)^{-(p/{s'})'}dy\bigg)^{\frac{1}{s'(p/{s'})'}}\\
&\leq C\bigg(\frac{1}{m(4\mathcal{B})}\int_{4\mathcal{B}}\omega^{s'}(y)^{{q}/{s'}}dy\bigg)^{-\frac{s'}{q}\cdot\frac{1}{s'}}\\
&=C\bigg(\frac{1}{m(4\mathcal{B})}\int_{4\mathcal{B}}\omega(y)^{q}dy\bigg)^{-1/{q}}.
\end{split}
\end{equation}
A direct calculation shows that
\begin{equation}\label{cal24}
\frac{1}{p'}+\frac{\,1\,}{q}=1-\frac{\,1\,}{p}+\frac{\,1\,}{q}=1-\frac{\alpha}{\,n\,}.
\end{equation}
In addition, since $\omega^{s'}\in A(p/{s'},q/{s'})$ with $s'\leq p<n/{\alpha}$ and $1/q=1/p-{\alpha}/n$, by Lemma \ref{lem5}, we have $\omega^q\in A_{\tau}$, where
\begin{equation}\label{samething}
\tau=1+\frac{q/{s'}}{(p/{s'})'}, ~\mbox{if}~ p>s',\quad\&\quad \tau=1, ~\mbox{if}~ p=s'.
\end{equation}
Hence, from \eqref{cal22}, \eqref{cal23} and \eqref{cal24}, it then follows that
\begin{equation*}
\begin{split}
\mathrm{III}&\leq C\|\Omega\|_{L^s(\mathbb{S}^{n-1})}
\frac{1}{\omega^q(\mathcal{B})^{{\gamma_*}/n}}\cdot\frac{m(4\mathcal{B})^{\frac{1}{p'}+\frac{1}{q}}}{m(\mathcal{B})^{1-\frac{\alpha}{n}}}\\
&\times\bigg(\int_{4\mathcal{B}}|f(y)|^p\omega(y)^pdy\bigg)^{1/p}\bigg(\int_{4\mathcal{B}}\omega(y)^{q}dy\bigg)^{-1/{q}}\\
&\leq C\|\Omega\|_{L^s(\mathbb{S}^{n-1})}\cdot
\frac{\omega^q(4\mathcal{B})^{\kappa/p-1/q}}{\omega^q(\mathcal{B})^{{\gamma_*}/n}}\|f\|_{\mathcal{M}^{p,\kappa}(\omega^p,\omega^q)}.
\end{split}
\end{equation*}
Note that ${\gamma_*}/n=\kappa/p-1/q$. Therefore, by part (1) of Lemma \ref{lem4}, we have
\begin{equation*}
\mathrm{III}\lesssim\|\Omega\|_{L^s(\mathbb{S}^{n-1})}\cdot
\frac{\omega^q(4\mathcal{B})^{{\gamma_*}/n}}{\omega^q(\mathcal{B})^{{\gamma_*}/n}}\|f\|_{\mathcal{M}^{p,\kappa}(\omega^p,\omega^q)}
\lesssim\|\Omega\|_{L^s(\mathbb{S}^{n-1})}\|f\|_{\mathcal{M}^{p,\kappa}(\omega^p,\omega^q)},
\end{equation*}
as desired. Now, let us turn to the estimate of the term IV. It is easy to see that
\begin{equation*}
\begin{split}
\mathrm{IV}
&\leq\frac{1}{\omega^q(\mathcal{B})^{{\gamma_*}/n}}\bigg(\frac{1}{m(\mathcal{B})}\int_{\mathcal{B}}
\bigg[\frac{1}{m(\mathcal{B})}\int_{\mathcal{B}}\big|T_{\Omega,\alpha}f_2(x)-T_{\Omega,\alpha}f_2(y)\big|\,dy\bigg]^{\ell}dx\bigg)^{1/{\ell}}\\
&\leq\frac{1}{\omega^q(\mathcal{B})^{{\gamma_*}/n}}\bigg(\frac{1}{m(\mathcal{B})}\int_{\mathcal{B}}\bigg[\frac{1}{m(\mathcal{B})}\int_{\mathcal{B}}\\
&\times\bigg\{\sum_{j=2}^\infty\int_{2^{j+1}\mathcal{B}\backslash 2^j\mathcal{B}}
\left|\frac{\Omega(x-z)}{|x-z|^{n-\alpha}}-\frac{\Omega(y-z)}{|y-z|^{n-\alpha}}\right|\cdot|f(z)|\,dz\bigg\}dy\bigg]^{\ell}dx\bigg)^{1/{\ell}}.
\end{split}
\end{equation*}
For each fixed $j\geq2$, it follows from H\"older's inequality with exponent $p/{s'}>1$ and the condition $\omega^{s'}\in A(p/{s'},q/{s'})$ that
\begin{equation*}
\begin{split}
&\bigg(\int_{2^{j+1}\mathcal{B}}|f(z)|^{s'}dz\bigg)^{1/{s'}}
=\bigg(\int_{2^{j+1}\mathcal{B}}|f(z)|^{s'}\omega(z)^{s'}\cdot \omega(z)^{-s'}dz\bigg)^{1/{s'}}\\
&\leq\bigg(\int_{2^{j+1}\mathcal{B}}|f(z)|^{p}\omega(z)^{p}dz\bigg)^{1/p}
\cdot\bigg(\int_{2^{j+1}\mathcal{B}}\omega^{s'}(z)^{-(p/{s'})'}dz\bigg)^{\frac{1}{s'(p/{s'})'}}\\
&\leq C\bigg(\int_{2^{j+1}\mathcal{B}}|f(z)|^{p}\omega(z)^{p}dz\bigg)^{1/p}\cdot m(2^{j+1}\mathcal{B})^{\frac{1}{s'(p/{s'})'}}
\bigg(\frac{1}{m(2^{j+1}\mathcal{B})}\int_{2^{j+1}\mathcal{B}}\omega^{s'}(z)^{{q}/{s'}}dz\bigg)^{-\frac{s'}{q}\cdot\frac{1}{s'}}\\
&= C\bigg(\int_{2^{j+1}\mathcal{B}}|f(z)|^{p}\omega(z)^{p}dz\bigg)^{1/p}
\cdot m(2^{j+1}\mathcal{B})^{\frac{1}{s'(p/{s'})'}+\frac{\,1\,}{q}}
\bigg(\int_{2^{j+1}\mathcal{B}}\omega(z)^{q}dz\bigg)^{-1/{q}}.
\end{split}
\end{equation*}
While for the case $s'=p$, from the condition $\omega^{p}\in A(1,q/{p})$, it follows that
\begin{equation*}
\begin{split}
&\bigg(\int_{2^{j+1}\mathcal{B}}|f(z)|^{s'}dz\bigg)^{1/{s'}}
=\bigg(\int_{2^{j+1}\mathcal{B}}|f(z)|^{p}\omega(z)^{p}\cdot \omega(z)^{-p}dz\bigg)^{1/{p}}\\
&\leq \bigg(\int_{2^{j+1}\mathcal{B}}|f(z)|^{p}\omega(z)^{p}dz\bigg)^{1/p}
\cdot\bigg[\underset{z\in 2^{j+1}\mathcal{B}}{\mathrm{ess\,sup}}\,\frac{1}{\omega^p(z)}\bigg]^{1/p}\\
&\leq C\bigg(\int_{2^{j+1}\mathcal{B}}|f(z)|^{p}\omega(z)^{p}dz\bigg)^{1/p}
\cdot m(2^{j+1}\mathcal{B})^{\frac{\,1\,}{q}}
\bigg(\int_{2^{j+1}\mathcal{B}}\omega(z)^{q}dz\bigg)^{-1/{q}}.
\end{split}
\end{equation*}
A direct computation gives us that
\begin{equation*}
\frac{1}{s'(p/{s'})'}+\frac{\,1\,}{q}=\frac{1}{s'}-\frac{\,1\,}{p}+\frac{\,1\,}{q}=-\frac{\,1\,}{s}+1-\frac{\alpha}{\,n\,},
\end{equation*}
and
\begin{equation*}
\frac{\,1\,}{q}=\frac{\,1\,}{p}-\frac{\alpha}{\,n\,}=-\frac{\,1\,}{s}+1-\frac{\alpha}{\,n\,},~~\mbox{when}~ s'=p.
\end{equation*}
Hence, we have
\begin{equation}\label{cal26}
\begin{split}
&\bigg(\int_{2^{j+1}\mathcal{B}}|f(z)|^{s'}dz\bigg)^{1/{s'}}\\
&\leq C\bigg(\int_{2^{j+1}\mathcal{B}}|f(z)|^{p}\omega(z)^{p}dz\bigg)^{1/p}
\cdot\frac{m(2^{j+1}\mathcal{B})^{-1/s+1-\alpha/n}}{\omega^q(2^{j+1}\mathcal{B})^{1/q}}  \\
&\leq C\|f\|_{\mathcal{M}^{p,\kappa}(\omega^p,\omega^q)}\omega^q(2^{j+1}\mathcal{B})^{\kappa/p-1/q}
\cdot m(2^{j+1}\mathcal{B})^{-1/s+1-\alpha/n}.
\end{split}
\end{equation}
Applying H\"{o}lder's inequality and Lemma \ref{lem3} together with \eqref{cal26}, we obtain that for any $x,y\in \mathcal{B}$,
\begin{equation*}
\begin{split}
&\int_{2^{j+1}\mathcal{B}\backslash 2^j\mathcal{B}}
\left|\frac{\Omega(x-z)}{|x-z|^{n-\alpha}}-\frac{\Omega(y-z)}{|y-z|^{n-\alpha}}\right|\cdot|f(z)|\,dz\\
&\leq\bigg(\int_{2^{j+1}\mathcal{B}\backslash 2^j\mathcal{B}}
\left|\frac{\Omega(x-z)}{|x-z|^{n-\alpha}}-\frac{\Omega(y-z)}{|y-z|^{n-\alpha}}\right|^sdz\bigg)^{1/s}
\cdot\bigg(\int_{2^{j+1}\mathcal{B}\backslash 2^j\mathcal{B}}|f(z)|^{s'}dz\bigg)^{1/{s'}}\\
&\leq C\cdot\big(2^jr\big)^{n/s-(n-\alpha)}
\bigg[\frac{1}{2^{j-1}}+\frac{1}{2^{j\beta}}
\int_{|x-x_0|/{(2^{j+1}r)}}^{|x-x_0|/{(2^jr)}}\frac{\omega_s(\delta)}{\delta^{1+\beta}}d\delta
+\frac{1}{2^{j\beta}}
\int_{|y-x_0|/{(2^{j+1}r)}}^{|y-x_0|/{(2^jr)}}\frac{\omega_s(\delta)}{\delta^{1+\beta}}d\delta\bigg]\\
&\times\bigg(\int_{2^{j+1}\mathcal{B}}|f(z)|^{s'}dz\bigg)^{1/{s'}}\\
&\leq C\|f\|_{\mathcal{M}^{p,\kappa}(\omega^p,\omega^q)}\omega^q(2^{j+1}\mathcal{B})^{\kappa/p-1/q}\\
&\times\bigg[\frac{1}{2^{j-1}}+\frac{1}{2^{j\beta}}
\int_{|x-x_0|/{(2^{j+1}r)}}^{|x-x_0|/{(2^jr)}}\frac{\omega_s(\delta)}{\delta^{1+\beta}}d\delta
+\frac{1}{2^{j\beta}}
\int_{|y-x_0|/{(2^{j+1}r)}}^{|y-x_0|/{(2^jr)}}\frac{\omega_s(\delta)}{\delta^{1+\beta}}d\delta\bigg].
\end{split}
\end{equation*}
Notice that $\kappa/p-1/q={\gamma_*}/n$. By part (2) of Lemma \ref{lem4}, we get
\begin{equation*}
\frac{\omega^q(2^{j+1}\mathcal{B})^{\kappa/p-1/q}}{\omega^q(\mathcal{B})^{{\gamma_*}/n}}
=\frac{\omega^q(2^{j+1}\mathcal{B})^{{\gamma_*}/{n}}}{\omega^q(\mathcal{B})^{{\gamma_*}/n}}
\leq C(2^{j+1})^{\tau\gamma_{*}},
\end{equation*}
where the number $\tau$ is the same as in \eqref{samething}. Therefore,
\begin{equation*}
\begin{split}
\mathrm{IV}&\lesssim\|f\|_{\mathcal{M}^{p,\kappa}(\omega^p,\omega^q)}
\sum_{j=2}^\infty
\frac{\omega^q(2^{j+1}\mathcal{B})^{{\gamma_*}/{n}}}{\omega^q(\mathcal{B})^{{\gamma_*}/n}}\\
&\times\bigg[\frac{1}{2^{j-1}}+\frac{1}{2^{j\beta}}
\int_{|x-x_0|/{(2^{j+1}r)}}^{|x-x_0|/{(2^jr)}}\frac{\omega_s(\delta)}{\delta^{1+\beta}}d\delta
+\frac{1}{2^{j\beta}}
\int_{|y-x_0|/{(2^{j+1}r)}}^{|y-x_0|/{(2^jr)}}\frac{\omega_s(\delta)}{\delta^{1+\beta}}d\delta\bigg]\\
&\lesssim\|f\|_{\mathcal{M}^{p,\kappa}(\omega^p,\omega^q)}\sum_{j=2}^\infty
\bigg[\frac{1}{2^{j(1-\tau\gamma_{*})}}+\frac{1}{2^{j(\beta-\tau\gamma_{*})}}
\int_{0}^{1}\frac{\omega_s(\delta)}{\delta^{1+\beta}}d\delta
+\frac{1}{2^{j(\beta-\tau\gamma_{*})}}
\int_{0}^{1}\frac{\omega_s(\delta)}{\delta^{1+\beta}}d\delta\bigg]\\
&\leq C\|f\|_{\mathcal{M}^{p,\kappa}(\omega^p,\omega^q)}\bigg[1+2\int_0^1\frac{\omega_s(\delta)}{\delta^{1+\beta}}d\delta\bigg],
\end{split}
\end{equation*}
where in the last step we have used the fact that $\tau\cdot\gamma_{*}<\beta\leq1$. Combining the above estimates for $\mathrm{III}$ and $\mathrm{IV}$ leads to the estimate \eqref{mainwh}. Finally, by taking the supremum over all balls $\mathcal{B}\subset\mathbb R^n$, we complete the proof of Theorem $\ref{thm2}$.
\end{proof}

\section*{Acknowledgments}
The author was supported by a grant from Xinjiang University under the project ``Real-Variable Theory of Function Spaces and Its Applications". This work was supported by the Natural Science Foundation of China (Grant No. XJEDU2020Y002 and 2022D01C407).

\bibliographystyle{elsarticle-harv}
\bibliography{<your bibdatabase>}

\end{document}